\documentclass[12pt]{article}
\usepackage{amssymb,latexsym,amsmath,enumerate,verbatim,amsfonts,amsthm}
\usepackage[active]{srcltx}

\numberwithin{equation}{section}
\newtheorem{Theorem}{Theorem}[section]
\newtheorem{Definition}{Definition}[section]
\newtheorem{Proposition}[Theorem]{Proposition}

\newtheorem{Lemma}{Lemma}[section]
\newtheorem{Corollary}[Theorem]{Corollary}

\makeatletter
 \def\@biblabel#1{#1.}
\makeatother

\newcommand{\x}{{\bf x}}
\newcommand{\y}{{\bf y}}
\newcommand{\z}{{\bf z}}
\newcommand{\q}{{\bf q}}

\newcommand{\e}{{\bf e}}
\newcommand{\0}{{\bf 0}}

\newcommand{\w}{{\bf w}}

\begin{document}
\title{Tensor Complementarity Problem and Semi-positive Tensors}

\author{Yisheng Song\thanks{Corresponding author. School of Mathematics and Information Science, Henan Normal University, XinXiang HeNan,  P.R. China, 453007.
Email: songyisheng1@gmail.com. The work was partially supported by the National Natural Science Foundation of P.R. China (Grant No. 11171094, 11271112, 61262026), NCET Programm of the Ministry of Education (NCET 13-0738),  science and technology programm of Jiangxi Education Committee (LDJH12088).}\quad \quad Liqun Qi\thanks{ Department of Applied Mathematics, The Hong Kong Polytechnic University, Hung Hom, Kowloon, Hong Kong. Email: maqilq@polyu.edu.hk. This author's work was supported by the Hong Kong Research Grant Council (Grant No. PolyU
502111, 501212, 501913 and 15302114).}}

\date{}

 \maketitle

\begin{abstract}
\noindent  
     The tensor complementarity problem $(\q, \mathcal{A})$ is to
	$$\mbox{ find } \x \in \mathbb{R}^n\mbox{ such that }\x \geq \0, \q + \mathcal{A}\x^{m-1} \geq \0, \mbox{ and }\x^\top (\q + \mathcal{A}\x^{m-1}) = 0.$$ We prove that  a real tensor $\mathcal{A}$ is a (strictly) semi-positive tensor if and only if the tensor complementarity problem $(\q, \mathcal{A})$ has a unique solution for $\q>\0$ ($\q\geq\0$), and  a symmetric real tensor is a (strictly) semi-positive tensor if and only if it is (strictly) copositive.  That is, for a  strictly copositive  symmetric  tensor $\mathcal{A}$,   the tensor complementarity problem $(\q, \mathcal{A})$ has a solution for all $\q \in \mathbb{R}^n$.
	\vspace{3mm}

\noindent {\bf Key words:}\hspace{2mm} Tensor complementarity,   strictly semi-positive,  strictly copositive, unique solution.  \vspace{3mm}

\noindent {\bf AMS subject classifications (2010):}\hspace{2mm}
47H15, 47H12, 34B10, 47A52, 47J10, 47H09, 15A48, 47H07.
  \vspace{3mm}
\end{abstract}

\section{Introduction}
It is well-known that the linear complementarity problem (LCP) is the first-order optimality conditions of quadratic programming, which has wide applications in  applied science and technology such as optimization and physical or economic equilibrium problems.
	Let $A = (a_{ij})$ be an $n \times n $ real matrix.  The linear complementarity problem, denoted by LCP $(\q,A)$,  is to
	\begin{equation}\label{eq:11}
	\mbox{ find } \z \in \mathbb{R}^n\mbox{ such that }\z \geq \0, \q + A\z \geq \0, \mbox{ and }\z^\top (\q + A\z) = 0.
	\end{equation}

By means of the linear complementarity problem, properties of (strictly) semi-monotone matrices were considered by Cottle and Dantzig \cite{CD68}, Eaves \cite{E71} and Karamardian \cite{K72}, see also Han, Xiu and Qi \cite{HXQ},  Facchinei and Pang\cite{FP11} and Cottle, Pang and Stone \cite{CPS}. A real matrix $A$ is said to be
	\begin{itemize}
		\item[(i)] {\bf semi-monotone (or semi-positive)} iff for each $\x\geq0$ and $\x\ne\0$, there exists an index $k\in I_n$ such that $$x_k>0\mbox{ and }\left(A \x\right)_k\geq0;$$
		\item[(ii)]  {\bf strictly semi-monotone (or strictly semi-positive)} iff for each $\x\geq\0$ and $\x\ne\0$, there exists an index $k\in I_n$ such that $$x_k>0\mbox{ and }\left(A\x\right)_k>0;$$
		\item[(iii)] {\bf copositive} iff
		$\x^\top A \x\geq 0$ for all  $\x\in \mathbb{R}^n_+$;
		\item[(iv)] {\bf strictly copositive} iff
		$\x^\top A \x>0$ for all  $\x\in \mathbb{R}^n_+\setminus\{\0\}$.\end{itemize}

Pang \cite{P79, P81} and Gowda \cite{S90} presented that some relations between the solution of  the LCP $(\q,A)$ and (strictly) semi-monotone.   Cottle \cite{C80} showed that each completely Q-matrix is a strictly semi-monotone matrix.    Eaves \cite{E71} gave an equivalent definition of strictly semi-monotone matrices using the linear complementarity problem. The concept of (strictly) copositive matrices is one of the most
important concept in applied mathematics and graph theory, which was
introduced by Motzkin \cite{TSM} in 1952. In the literature, there
are extensive discussions on such matrices \cite{HH}, \cite{DM}, \cite{HV}.

The nonlinear complementarity problem has been systematically studied  in
the mid-1960s and has developed into a very fruitful discipline in the field of mathematical programming, that included a multitude of interesting connections to numerous disciplines and a wide range of important applications in engineering and economics. We will study the tensor complementarity problem,  a special structured nonlinear complementarity problem.  Let $\mathcal{A} = (a_{i_1\cdots i_m})$ be a real $m$th order $n$-dimensional tensor (hypermatrix).
	The tensor complementarity problem, denoted by TCP $(\q, \mathcal{A})$, is to
	\begin{equation}\label{eq:12} \mbox{ find } \x \in \mathbb{R}^n\mbox{ such that }\x \geq \0, \q + \mathcal{A}\x^{m-1} \geq \0, \mbox{ and }\x^\top (\q + \mathcal{A}\x^{m-1}) = 0.\end{equation}
Clearly, the tensor complementarity problem  is the first-order optimality conditions of the homogeneous polynomial optimization problem, which may be referred to as a direct and natural extension of the linear complementarity problem. The tensor complementarity problem  TCP $(\q, \mathcal{A})$ is a specially  structured nonlinear  complementarity problem,  and so, the TCP  $(\q, \mathcal{A})$ has its particular and nice properties other than ones of the classical NCP(F).

In this paper,  we will study some relationships between the unique solution of the tensor complementarity problem and  (strictly) semi-positive tensors. We will prove that a symmetric tensor is (strictly) semi-positive if and only if it is (strictly) copositive.

In Section 2, we will give some  definitions and basic conclusions. We will show that all diagonal entries of a semi-positive tensor are nonnegative, and all diagonal entries of a strictly semi-positive tensor are positive.

In Section 3, We will prove that  a real tensor $\mathcal{A}$ is a  semi-positive tensor if and only if the TCP $(\q, \mathcal{A})$ has no non-zero vector solution for $\q>\0$ and  a real tensor $\mathcal{A}$ is a strictly semi-positive tensor if and only if the TCP $(\q, \mathcal{A})$ has no non-zero vector solution for $\q\geq\0$.  We show that a symmetric real tensor is  semi-positive if and only if it is  copositive and a symmetric real tensor is a strictly semi-positive if and only if it is strictly copositive.  That is, for a  strictly copositive  symmetric  tensor $\mathcal{A}$,   the TCP $(\q, \mathcal{A})$ has a solution for all $\q \in \mathbb{R}^n$.	
	
\section{Preliminaries}
\hspace{4mm}
Throughout this paper, we use small letters $x, y, v, \alpha, \cdots$, for scalars, small bold
letters $\x, \y,  \cdots$, for vectors, capital letters $A, B,
\cdots$, for matrices, calligraphic letters $\mathcal{A}, \mathcal{B}, \cdots$, for
tensors.  All the tensors discussed in this paper are real.  Let $I_n := \{ 1,2, \cdots, n \}$, and $\mathbb{R}^n:=\{(x_1, x_2,\cdots, x_n)^\top;x_i\in  \mathbb{R},  i\in I_n\}$, $\mathbb{R}^n_{+}:=\{x\in \mathbb{R}^n;x\geq\0\}$, $\mathbb{R}^n_{-}:=\{\x\in \mathbb{R}^n;x\leq\0\}$, $\mathbb{R}^n_{++}:=\{\x\in \mathbb{R}^n;x>\0\}$, where $\mathbb{R}$ is the set of real numbers, $\x^\top$ is the transposition of a vector $\x$, and $\x\geq\0$ ($\x>\0$) means $x_i\geq0$ ($x_i>0$) for all $i\in I_n$.

In 2005, Qi \cite{Qi} introduced the concept of
positive (semi-)definite symmetric tensors.   A real $m$th order $n$-dimensional tensor (hypermatrix) $\mathcal{A} = (a_{i_1\cdots i_m})$ is a multi-array of real entries $a_{i_1\cdots
	i_m}$, where $i_j \in I_n$ for $j \in I_m$. Denote the set of all
real $m$th order $n$-dimensional tensors by $T_{m, n}$. Then $T_{m,
	n}$ is a linear space of dimension $n^m$. Let $\mathcal{A} = (a_{i_1\cdots
	i_m}) \in T_{m, n}$. If the entries $a_{i_1\cdots i_m}$ are
invariant under any permutation of their indices, then $\mathcal{A}$ is
called a {\bf symmetric tensor}.  Denote the set of all real $m$th
order $n$-dimensional tensors by $S_{m, n}$. Then $S_{m, n}$ is a
linear subspace of $T_{m, n}$.  We
denote the zero tensor in $T_{m, n}$ by $\mathcal{O}$.   Let $\mathcal{A} =(a_{i_1\cdots i_m}) \in
T_{m, n}$ and $\x \in \mathbb{R}^n$. Then $\mathcal{A} \x^{m-1}$ is a vector in $\mathbb{R}^n$ with
its $i$th component as
$$\left(\mathcal{A} \x^{m-1}\right)_i: = \sum_{i_2, \cdots, i_m=1}^n a_{ii_2\cdots
	i_m}x_{i_2}\cdots x_{i_m}$$ for $i \in I_n$.
 Then $\mathcal{A} \x^m$ is a homogeneous
 polynomial of degree $m$, defined by
 $$\mathcal{A} \x^m:= \x^\top(\mathcal{A} \x^{m-1})= \sum_{i_1,\cdots, i_m=1}^n a_{i_1\cdots i_m}x_{i_1}\cdots
 x_{i_m}.$$
 A tensor $\mathcal{A} \in T_{m, n}$ is called {\bf positive
semi-definite} if for any vector $\x \in \mathbb{R}^n$, $\mathcal{A} \x^m \ge 0$,
and is called {\bf positive definite} if for any nonzero vector $\x
\in \mathbb{R}^n$, $\mathcal{A} \x^m > 0$.
 Recently, miscellaneous structured tensors are widely studied, for example,
Zhang, Qi and Zhou \cite{ZQZ} and
 Ding, Qi and Wei \cite{DQW} for M-tensors, Song and Qi \cite{SQ-15} for P-(P$_0$)tensors and B-(B$_0$)tensors, Qi and Song \cite{QS} for  B-(B$_0$)tensors,  Song and Qi \cite{SQ1} for infinite and finite dimensional Hilbert tensors, Song and Qi \cite{SQ} for E-eigenvalues of weakly symmetric nonnegative tensors and so on.

 Recently, Song and Qi \cite{SQ2015} extended the concepts of  (strictly) semi-positive matrices and the linear complementarity problem to  (strictly) semi-positive tensors and the tensor complementarity problem, respectively. Moreover, some nice properties of those concepts were obtained.
 \begin{Definition} \label{d21}\em
	Let $\mathcal{A}  = (a_{i_1\cdots i_m}) \in T_{m, n}$.   $\mathcal{A}$ is said to be\begin{itemize}
\item[(i)] {\bf semi-positive} iff for each $\x\geq0$ and $\x\ne\0$, there exists an index $k\in I_n$ such that $$x_k>0\mbox{ and }\left(\mathcal{A} \x^{m-1}\right)_k\geq0;$$
\item[(ii)]  {\bf strictly semi-positive} iff for each $\x\geq\0$ and $\x\ne\0$, there exists an index $k\in I_n$ such that $$x_k>0\mbox{ and }\left(\mathcal{A} \x^{m-1}\right)_k>0;$$
\item[(iii)]  {\bf  Q-tensor} iff the TCP $(\q, \mathcal{A})$ has a  solution for all $\q \geq \mathbb{R}^n$.
\end{itemize}
\end{Definition}

\begin{Lemma}\label{le21}\em(Song and Qi \cite[Corollary 3.3, Theorem 3.4]{SQ2015}) Each strictly semi-positive tensor must be  a Q-tensor.
\end{Lemma}

\begin{Proposition}\em \label{pro:21}
	Let $\mathcal{A}\in T_{m,n}$. Then
	\begin{itemize}
		\item[(i)] $a_{ii\cdots i}\geq0$ for all $i\in I_n$ if  $\mathcal{A}$ is semi-positive;
		\item[(ii)] $a_{ii\cdots i}>0$ for all $i\in I_n$ if  $\mathcal{A}$ is strictly semi-positive;
		\item[(iii)] there exists $k\in I_n$ such that  $\sum\limits_{i_2, \cdots, i_m=1}^n a_{ki_2\cdots
			i_m}\geq0$ if  $\mathcal{A}$ is semi-positive;
		\item[(vi)] there exists $k\in I_n$ such that  $\sum\limits_{i_2, \cdots, i_m=1}^n a_{ki_2\cdots
			i_m}>0$ if  $\mathcal{A}$ is strictly semi-positive.
		\end{itemize}
\end{Proposition}
\begin{proof}
	It follows from Definition \ref{d21} that we can obtain (i) and (ii) by taking $$\x^{(i)}=(0,\cdots,1,\cdots,0)^\top,\ i\in I_n$$
	where $1$ is the ith component $x_i$.
	Similarly, choose $\x=\e=(1,1,\cdots,1)^\top$, then we obtain (iii) and (vi) by Definition \ref{d21}.
\end{proof}

\begin{Definition} \label{d22}\em
	Let $\mathcal{A}  = (a_{i_1\cdots i_m}) \in T_{m, n}$.   $\mathcal{A}$ is said to be\begin{itemize}
		\item[(i)] {\bf copositive } if $\mathcal{A}\x^m\geq0$ for all $\x\in \mathbb{R}^n_+$;
		\item[(ii)] {\bf strictly copositive} if  $\mathcal{A}\x^m>0$ for all $\x\in \mathbb{R}^n_+\setminus\{\0\}$.
	\end{itemize}
\end{Definition}

The concept of (strictly) copositive tensors was first introduced and used by Qi in \cite{Qi1}.  Song and Qi \cite{SQ15} showed their equivalent definition and some special structures. The following lemma is one of the structure conclusions  of (strictly) copositive tensors in \cite{SQ15}.

\begin{Lemma}\em { (\cite[Proposition 3.1]{SQ15})}\label{le:22}
 Let $\mathcal{A}$ be a symmetric tensor of order $m$ and dimension $n$. Then
\begin{itemize}
\item[(i)] $\mathcal{A}$ is copositive if and only if  $\mathcal{A}x^m\geq0$ for all $x\in \mathbb{R}^n_+$ with $\|x\|=1$;
\item[(ii)] $\mathcal{A}$ is strictly copositive if and only if $\mathcal{A}x^m>0$ for all $x\in \mathbb{R}^n_+$ with $\|x\|=1$.\end{itemize}
\end{Lemma}
\begin{Definition} \label{d23}\em Let $\mathcal{A}  = (a_{i_1\cdots i_m}) \in T_{m, n}$. In homogeneous polynomial $\mathcal{A}\x^m$, if we let some (but not all) $x_i$ be zero, then we have a less variable homogeneous polynomial, which defines a lower dimensional tensor. We call such a lower dimensional tensor a {\bf principal sub-tensor} of $\mathcal{A}$, i.e.,  an $m$-order $r$-dimensional principal sub-tensor $\mathcal{B}$  of an $m$-order $n$-dimensional tensor $\mathcal{A}$ consists of $r^m$ entries in $\mathcal{A} = (a_{i_1\cdots i_m})$: for any set $\mathcal{N}$ that composed of $r$ elements in $ \{1,2,\cdots , n\}$,
$$\mathcal{B} = (a_{i_1\cdots i_m}),\mbox{ for all } i_1, i_2, \cdots, i_m\in \mathcal{N}.$$
\end{Definition}
The concept were first introduced and used by Qi \cite{Qi} to the higher order symmetric tensor.
It follows from Definition \ref{d21} that the following Proposition is  obvious.

\begin{Proposition} \em Let $\mathcal{A}  = (a_{i_1\cdots i_m}) \in T_{m, n}$. Then
\begin{itemize}
\item[(i)] each principal sub-tensor of a semi-positive tensor is  semi-positive;
\item[(ii)] each principal sub-tensor of a strictly semi-positive tensor is strictly semi-positive.
\end{itemize}
\end{Proposition}

  Let $N\subset I_n=\{1, 2,  \cdots, n\}$.  We denote the principal sub-tensor of   $\mathcal{A}$ by $\mathcal{A}^{|N|}$, where $|N|$ is  the cardinality of $N$.  So, $\mathcal{A}^{|N|}$ is a tensor of order $m$ and dimension $|N|$ and the principal sub-tensor $\mathcal{A}^{|N|}$  is just $\mathcal{A}$ itself when $N=I_n=\{1, 2,  \cdots, n\}$.

\section{Main results}

In this section, we will prove that   a real tensor $\mathcal{A}$ is a (strictly) semi-positive tensor if and only if the tensor complementarity problem $(\q, \mathcal{A})$ has a unique solution for $\q>\0$ ($\q\geq\0$).
\begin{Theorem}\label{th:31}\em Let $\mathcal{A}  = (a_{i_1\cdots i_m}) \in T_{m, n}$. The following statements are equivalent:
\begin{itemize}
  \item[(i)]  $\mathcal{A}$ is semi-positive.
  \item[(ii)] The TCP $(\q, \mathcal{A})$ has a unique solution for every $\q > \0$.
  \item[(iii)] For every index set $N\subset I_n$, the system
\begin{equation}\label{eq:32} \mathcal{A}^{|N|}(\x^N)^{m-1}<\0, \ \x^N\geq \0 \end{equation}
has no solution, where $\x^N\in \mathbb{R}^{|N|}.$
\end{itemize}
 \end{Theorem}

 \begin{proof} (i) $\Rightarrow$ (ii). Since $\q > \0$, it is obvious that $\0$ is a  solution of TCP $(\q, \mathcal{A})$. Suppose that there exists a vector $\q'>\0$ such that TCP $(\q', \mathcal{A})$ has non-zero vector solution $\x.$  Since $\mathcal{A}$ is semi-positive, there is an index $k\in I_n$ such that $$x_k>0\mbox{ and } \left(\mathcal{A}\x^{m-1}\right)_k\geq 0.$$
Then $q'_k+\left(\mathcal{A}\x^{m-1}\right)_k>0$,  and so $$\x^\top (\q' + \mathcal{A}\x^{m-1})>0.$$
This   contradicts the assumption that $\x$ solves TCP $(\q', \mathcal{A})$. So the TCP $(\q, \mathcal{A})$ has a unique solution $\0$ for every $\q > \0$.\\

(ii) $\Rightarrow$ (iii).  Suppose that  there is an index set $N$ such that the system (\ref{eq:32}) has a solution $\bar{\x}^{N}$. Clearly, $\bar{\x}^N\ne\0$.  Let $\bar{\x}=(\bar{x}_1, \bar{x}_2,\cdots,\bar{x}_n)^\top $ with
$$\bar{x}_i=\begin{cases}\bar{\x}^N_i, i\in N\\
0,\ \ i\in I_n\setminus N.
\end{cases}$$
Choose $\q=(q_1, q_2,\cdots,q_n)^\top $ with $$\begin{cases} q_i =-\left(\mathcal{A}^{|N|}(\bar{\x}^{N})^{m-1}\right)_i=-\left(\mathcal{A}\bar{\x}^{m-1}\right)_i, i\in N\\
q_i >\max\{0, -\left(\mathcal{A}\bar{\x}^{m-1}\right)_i\}\ \ i\in I_n\setminus N.
\end{cases}$$
So, $\q>\0$ and $\bar{\x}\ne \0$.  Then $\bar{\x}$ solves the TCP $(\q, \mathcal{A}).$ This contradicts (ii).\\

(iii) $\Rightarrow$ (i). For each $\x\in\mathbb{R}^n_+$ and $\x\ne \0$, we may assume that $\x=(x_1,x_2,\cdots,x_n)^\top$ with for some $N\subset I_n$, $$\begin{cases}x_i>0, i\in N\\
x_i=0,i\in I_n\setminus N.
\end{cases}$$ Since the system (\ref{eq:32}) has no solution,  there exists an index $k\in N\subset I_n$ such that $$x_k>0\mbox{ and } \left(\mathcal{A}\x^{m-1}\right)_k\geq0,$$ and hence $\mathcal{A}$ is semi-positive.
\end{proof}
Using the same proof as that of Theorem \ref{th:31} with appropriate changes in the
inequalities. We can obtain the following conclusions about the strictly semi-positive tensor.

\begin{Theorem}\label{th:32}\em Let $\mathcal{A}  = (a_{i_1\cdots i_m}) \in T_{m, n}$. The following statements are equivalent:
	\begin{itemize}
		\item[(i)]  $\mathcal{A}$ is strictly semi-positive.
		\item[(ii)] The TCP $(\q, \mathcal{A})$ has a unique solution for every $\q \geq \0$.
		\item[(iii)] For every index set $N\subset I_n$, the system
		\begin{equation}\label{eq:33} \mathcal{A}^{|N|}(\x^N)^{m-1}\leq\0, \ \x^N\geq \0,\ \x^N\ne\0 \end{equation}
		has no solution.
	\end{itemize}
\end{Theorem}

Now we give the following main results by means of the concept of principal sub-tensor.
\begin{Theorem}\label{th:33}\em Let $\mathcal{A}$ be a symmetric tensor of order $m$ and dimension $n$.  Then $\mathcal{A}$ is semi-positive if and only if  it is copositive.
\end{Theorem}

\begin{proof} If $\mathcal{A}$ is copositive, then 	\begin{equation}\label{eq:34}\mathcal{A}\x^m=\x^\top \mathcal{A}\x^{m-1}\geq0\mbox{ for all }\x\in \mathbb{R}^n_+.\end{equation}
So $\mathcal{A}$ must be semi-positive.
In fact, suppose not. Then for all $k\in I_n $ such that
	$$\left(\mathcal{A}\x^{m-1}\right)_k<0 \mbox{ when } x_k>0.$$ Then we have $$\mathcal{A}\x^m=\x^\top \mathcal{A}\x^{m-1}=\sum_{k=1}^{n}x_k\left(\mathcal{A}\x^{m-1}\right)_k<0,$$
	 which contradicts (\ref{eq:34}).\\
	
Now we show the necessity. Let $$ S=\{\x\in\mathbb{R}^{n}_+;\sum\limits_{i=1}^{n}x_i=1\}\mbox{ and }F(\x)=\mathcal{A}\x^m=\x^\top\mathcal{A}\x^{m-1}.$$
  Obviously, $F:S\to \mathbb{R}$  is continuous on the set $S$.  Then there exists $\tilde{\y}\in S$ such that \begin{equation}\label{eq:36}
\mathcal{A}\tilde{\y}^m=\tilde{\y}^\top \mathcal{A}\tilde{\y}^{m-1}=F(\tilde{\y})=\min_{\x\in S}F(\x)=\min_{\x\in S}\x^\top\mathcal{A}\x^{m-1}=\min_{\x\in S}\mathcal{A}\x^m.
\end{equation}
Since $\tilde{\y}\geq0$ with $\tilde{\y}\neq0$, we may assume that
$$\tilde{\y}=(\tilde{y}_1,\tilde{y}_2,\cdots,\tilde{y}_l,0,\cdots,0)^T\ (\tilde{y}_i>0\mbox{ for }i=1,\cdots,l, 1\leq l\leq n). $$
Let $\tilde{\w}=(\tilde{y}_1,\tilde{y}_2,\cdots,\tilde{y}_l)^T$ and let $\mathcal{B}$ be a principal sub-tensor that obtained from $\mathcal{A}$ by the polynomial $\mathcal{A}x^m$ for $\x=(x_1,x_2,\cdots,x_l,0,\cdots,0)^T$. Then \begin{equation} \label{eq:25}
\tilde{\w}\in\mathbb{R}^l_{++},\  \sum\limits_{i=1}^l\tilde{y}_i = 1\mbox{ and }F(\tilde{\y})=\mathcal{A}\tilde{\y}^m=\mathcal{B}\tilde{\w}^m=\min_{\x\in S}\mathcal{A}\x^m.
\end{equation}

Let $\x=(z_1,z_2,\cdots,z_l,0,\cdots,0)^T\in\mathbb{R}^n$ for all $\z=(z_1,z_2,\cdots,z_l)^T\in\mathbb{R}^l$ with $\sum\limits_{i=1}^lz_i = 1$. Clearly, $\x\in S$, and hence, by  (\ref{eq:25}), we have $$F(x)=\mathcal{A}\x^m=\mathcal{B}\z^m\geq F(\tilde{\y})=\mathcal{A}\tilde{\y}^m=\mathcal{B}\tilde{\w}^m.$$ Since $\tilde{\w}\in\mathbb{R}^l_{++}$,
$\tilde{\w}$ is a local minimizer of the following optimization problem
$$ \begin{aligned}
    \min_{\z\in \mathbb{R}^l} &\ \mathcal{B}\z^m\\
     s.t. &\ \sum\limits_{i=1}^lz_i = 1.
      \end{aligned}$$
So, the standard KKT conditions
implies that there exists $\mu\in \mathbb{R}$ such that
$$\nabla(\mathcal{B}\z^m-\mu(\sum\limits_{i=1}^lz_i -1))|_{\z=\tilde{\w}}= m\mathcal{B}\tilde{\w}^{m-1}-\mu \e=0,$$
where $\e=(1,1,\cdots,1)^\top$, and hence $$\mathcal{B}\tilde{\w}^{m-1}=\frac{\mu}m \e.$$

Let $\lambda=\frac{\mu}m$.  Then $$\mathcal{B}\tilde{\w}^{m-1}=(\lambda,\lambda,\cdots,\lambda)^\top\in \mathbb{R}^l,$$ and so $$\mathcal{B}\tilde{\w}^m=\tilde{\w}^\top\mathcal{B}\tilde{\w}^{m-1}=\lambda\sum\limits_{i=1}^l\tilde{y}_i =\lambda.$$It follows from (\ref{eq:25}) that $$\mathcal{A}\tilde{\y}^m=\tilde{\y}^\top  \mathcal{A}\tilde{\y}^{m-1}=\mathcal{B}\tilde{\w}^m=\min_{\x\in S}\mathcal{A}\x^m =\lambda.$$
Thus, for all $\tilde{y}_k>0$, we have
$$
\left(\mathcal{A}\tilde{\y}^{m-1}\right)_k=\left(\mathcal{B}\tilde{\w}^{m-1}\right)_k=\lambda.$$

Since $\mathcal{A}$ is semi-positive,  for $\tilde{\y}\geq\0$ and $\tilde{\y}\ne\0$,  there exists an index $k\in I_n$ such that $$\tilde{\y}_k>0\mbox{ and }\left( \mathcal{A}\tilde{\y}^{m-1}\right)_k\geq0. $$
and hence,  $\lambda\geq0$. Consequently, we have
$$\min_{\x\in S}\mathcal{A}\x^m=\mathcal{A}\tilde{\y}^m=\lambda\geq0.$$
It follows from Lemma \ref{le:22} that $\mathcal{A}$ is copositive.
  The theorem is proved.
 \end{proof}

Using the same proof as that of Theorem \ref{th:33} with appropriate changes in the
inequalities, we can obtain the following conclusions about the strictly copositive tensor.

 \begin{Theorem}\label{th:34}\em Let $\mathcal{A}  = (a_{i_1\cdots i_m}) \in S_{m,n}$. Then $\mathcal{A}$ is strictly semi-positive if and only if  it is strictly copositive.
\end{Theorem}
By Lemma \ref{le21} and Theorem \ref{th:34}, the following conclusion is obvious.
 \begin{Corollary}\label{co:35}\em Let  $\mathcal{A}  = (a_{i_1\cdots i_m})\in S_{m,n}$ be strictly copositive. Then the tensor complementarity problem TCP $(\q, \mathcal{A})$, 	$$\mbox{finding $\x \in \mathbb{R}^n$ such that } \x \geq \0, \q + \mathcal{A}\x^{m-1} \geq \0, \mbox{ and }\x^\top (\q + \mathcal{A}\x^{m-1}) = 0$$ has a solution for all $\q \in \mathbb{R}^n$.
\end{Corollary}





\begin{thebibliography}{}
\bibitem{CD68} Cottle,  R.W., Randow, R.V.: Complementary pivot theory of mathematical programming. Linear Algebra and its Applications 1,  103-125  (1968).
\bibitem{E71} Eaves, B.C.: The linear complementarity problem. Management Science 17, 621-634  (1971).
\bibitem{K72} Karamardian,  S.: The complementarity problem. Mathematical Programming 2, 107-129  (1972).

\bibitem{HXQ} Han, J.Y., Xiu, N.H., Qi, H.D.: Nonlinear complementary Theory and Algorithm.  Shanghai Science and Technology Press, Shanghai. (2006) (in Chinese).
\bibitem{FP11}Facchinei, F., Pang, J.S.: Finite-Dimensional Variational Inequalities and Complementarity Problems: Volume I, Springer-Verlag New York Inc. (2011).
\bibitem{CPS} Cottle, R.W.,  Pang, J.S.,  Stone, R.E.: The Linear Complementarity Problem. Academic Press, Boston (1992)
\bibitem{P79}Pang, J.S.:  On Q-matrices. Mathematical Programming 17, 243-247  (1979).
\bibitem{P81}Pang, J.S.:  A unification of  two classes of Q-matrices. Mathematical Programming 20, 348-352  (1981).
\bibitem{S90} Gowda, M.S.: On Q-matrices. Mathematical Programming 49, 139-141  (1990).
\bibitem{C80} Cottle,  R.W.: Completely Q-matrices. Mathematical Programming 19, 347-351  (1980).

\bibitem{TSM}Motzkin, T.S.: Quadratic forms, National Bureau of Standards Report,1818 (1952) 11-12.

\bibitem{HH}Haynsworth, E., Hoffman, A.J.: Two remarks on copositive matrices, Linear Algebra Appl. 2(1969) 387-392.

\bibitem{DM} Martin, D.H.: Copositlve matrices and definiteness of quadratic forms
	subject to homogeneous linear inequality constraints,  Linear Algebra Appl., 35(1981) 227-258.

\bibitem{HV}V${\ddot{a}}$liaho, H.: Criteria for copositive matrices, Linear Algebra Appl. 81(1986) 19-34.



\bibitem{Qi} Qi, L.: Eigenvalues of a real supersymmetric tensor.   J. Symbolic Comput. 40, 1302-1324 (2005)
\bibitem{ZQZ}
Zhang, L., Qi, L., Zhou,  G.: M-tensors and some applications.   SIAM J. Matrix Anal. Appl.  35(2), 437-452 (2014)
\bibitem{DQW} Ding, W., Qi, L., Wei, Y.: M-tensors and nonsingular
M-tensors.  Linear Algebra Appl.   439, 3264-3278 (2013)
\bibitem{SQ-15} Song, Y., Qi, L.; Properties of Some Classes of Structured Tensors. to appear in:  J. Optim. Theory Appl., 2014 DOI 10.1007/s10957-014-0616-5
\bibitem{QS} Qi, L., Song, Y.: An even order symmetric B tensor is positive definite.  Linear Algebra Appl. 457, 303-312 (2014)
\bibitem{SQ1} Song, Y., Qi, L.; Infinite and finite dimensional Hilbert tensors. Linear Algebra Appl. 451, 1-14 (2014)

\bibitem{SQ} Song, Y., Qi, L.: Spectral properties of positively homogeneous operators induced by higher order tensors.  SIAM J. Matrix Anal. Appl.
 34(4), 1581-1595 (2013).
 \bibitem{SQ2015} Song, Y., Qi, L.;  Properties of Tensor Complementarity Problem and Some Classes of Structured Tensors arXiv:1412.0113v1, Nov. 2014.
\bibitem{Qi1} Qi, L.: Symmetric nonegative tensors and copositive tensors. Linear Algebra Appl. 439, 228-238 (2013)
\bibitem{SQ15} Song, Y., Qi, L.; Necessary and sufficient conditions for copositive tensors. Linear and Multilinear Algebra,  63(1), 120-131 (2015).

\end{thebibliography}
\end{document}